\documentclass[12pt, reqno]{amsart}
\usepackage{amsmath, amsthm, amscd, amsfonts, amssymb, graphicx, color}
\usepackage[bookmarksnumbered, colorlinks, plainpages]{hyperref}
\hypersetup{colorlinks=true,linkcolor=red, anchorcolor=green, citecolor=cyan, urlcolor=red, filecolor=magenta, pdftoolbar=true}

\newtheorem{theorem}{Theorem}[section]
\newtheorem{lemma}[theorem]{Lemma}
\newtheorem{proposition}[theorem]{Proposition}
\newtheorem{corollary}[theorem]{Corollary}
\theoremstyle{definition}
\newtheorem{definition}[theorem]{Definition}

\newtheorem{problem}[theorem]{Problem}
\theoremstyle{remark}
\newtheorem{remark}[theorem]{Remark}

\newcommand{\E}{{\mathbb E}}

\numberwithin{equation}{section}

\begin{document}
\setcounter{page}{1}

\title[Martingale Hardy spaces]{Martingale Hardy spaces with variable exponents}

\author[Y. Jiao, D. Zhou, Z. Hao, W. Chen]{Yong Jiao$^1$$^{*}$, Dejian Zhou$^1$, Zhiwei Hao$^1$ and Wei Chen$^2$}

\address{$^{1}$ School of Mathematics and Statistics, Central South University, Changsha 410083, China.}
\email{\textcolor[rgb]{0.00,0.00,0.84}{jiaoyong@csu.edu.cn;
zhoudejian@csu.edu.cn; zw\underline{ }hao8@163.com}}

\address{$^{2}$ School of Mathematical Sciences, Yangzhou University,
Yangzhou 225002, China.}
\email{\textcolor[rgb]{0.00,0.00,0.84}{weichen@yzu.edu.cn}}



\date{Received: xxxxxx; Revised: yyyyyy; Accepted: zzzzzz.
\newline \indent $^{*}$ Corresponding author}

\begin{abstract}
In this paper, we introduce Hardy spaces with variable exponents defined on a
probability space and develop the martingale theory of variable
Hardy spaces. We prove the weak type and strong type
inequalities on Doob's maximal operator and get a  $(1,p(\cdot),\infty)$-atomic decomposition
for Hardy martingale spaces associated with conditional square functions. As
applications, we obtain a dual theorem and the John-Nirenberg
inequalities in the frame of variable exponents. The key ingredient is that we find a condition with probabilistic characterization of $p(\cdot)$ to replace the so-called log-H\"{o}lder continuity condition  in $\mathbb {R}^n.$
\end{abstract} \maketitle

\section{Introduction}
Let $p(\cdot): \mathbb{R}^n\rightarrow (0,\infty)$ be a measurable
function such that $0<\inf_{x\in \mathbb{R}^n}p(x)\leq \sup_{x\in
\mathbb{R}^n}p(x)<\infty.$ The space $L^{p(\cdot)}(\mathbb{R}^n)$,
the Lebesgue space with variable exponent $p(\cdot),$ is defined as
the set of all measurable functions $f$ such that for some
$\lambda>0$ $$\int_{\mathbb{R}^n}\left(\frac{|f(x)|}{\lambda }\right)^{p(x)}dx<\infty,$$ with
$$\|f\|_{p(\cdot)}:=\inf\Bigg\{\lambda>0:\int_{\mathbb{R}^n}\Bigg(\frac{|f(x)|}{\lambda}\Bigg)^{p(x)}dx\leq1\Bigg\}.$$
Then $\big(L^{p(\cdot)}, \|\cdot\|_{p(\cdot)}\big)$ is a quasi-normed space.
Such Lebesgue spaces were introduced by Orlicz \cite{orlicz} in 1931
and studied by O. Kov\`{a}\u{c}ik and J. R\'{a}kosn\'{i}k \cite{KR},
X. Fan and D. Zhao \cite{FZ} and others. We refer to two new monograghs
\cite{CF} and \cite{DH} for the recent progress on Lebesgue spaces
with variable exponents and some applications in PDEs and
variational integrals with nonstandard growth conditions. We also
note that in the recent years more attention was turned to the study
of function spaces with variable exponent in harmonic analysis; see
for instance \cite{ah,cruz,cruz2,d2,NS2,sawano,yangdachun}. Let $\Omega \subset \mathbb{R}^n$. We say that a function $p(\cdot):\Omega\rightarrow \mathbb{R}$ is locally log-H\"{o}lder continuous on $\Omega$ if there exists $c_1>0$ such that
\begin{equation}\label{log}
|p(x)-p(y)|\leq \frac{c_1}{\log(e+1/|x-y|)}
\end{equation}
for all $x,y\in \Omega$. Heavily relying on the
so-called log-H\"{o}lder continuity conditions on the variable
exponent functions, in the pioneering work \cite{diening}, Diening
proved that the Hardy-Littlewood maximal operator is bounded on
$L^{p(\cdot)}(\mathbb{R}^n)$. An example in \cite{pick} showed that
log-H\"{o}lder continuity of $p(x)$ is essentially the optimal
condition when the maximal operator is bounded on variable
exponent Lebesgue spaces defined on Euclidean spaces (even in the
doubling metric measure spaces; see \cite{HHM}). We refer to \cite{al} for more questions related to the
maximal operator in variable $L^{p(\cdot)}.$

Although variable exponent Lebesgue spaces on Euclidean space have
attracted a steadily increasing interest over the last couple of
years, the variable exponent framework has not yet been applied to
the probability space setting.  The purpose of the present paper is to introduce Hardy martingale
 spaces with variable exponent, and to develop the martingale theory of variable Hardy spaces. To the best of our knowledge, our paper is the first treating from this perspective.
For a convenience, we
first fix some notations. Let $(\Omega,\mathcal {F},\mathbb{P})$ be
a complete probability space and $\mathcal{P}=\mathcal{P}(\Omega)$
denote the collection of all measurable functions $p(\cdot)$ :
$\Omega\longrightarrow(0,\infty)$ which is called a variable
exponent. For a measurable set $A\subset\Omega$, we denote
$$p_+(A)=\sup_{x\in A}p(x),~~~~~p_-(A)=\inf_{x\in A}p(x) $$
and
$${p_+=p_+(\Omega)},~~~~~p_-=p_-(\Omega).$$ Compared with Euclidean
space $\mathbb{R}^n$,  the probability space $\Omega$ has no natural metric structure. The main difficulty is how to overcome the log-H\"{o}lder continuity \eqref{log} when $p(x)$ is defined on a probability space $(\Omega,\mathcal {F},\mathbb{P})$.

The first aim of this paper is to discuss the weak type and strong type inequalities about Doob's
maximal operator. Aoyama \cite{HA} proved that
Doob's maximal inequality is true under some conditions. Namely, if
$1\leq p(\cdot)<\infty$ and there exists a constant $C$
such that
\begin{equation}\label{f1}\frac{1}{p(\cdot)}\leq C \E\Big(\frac{1}{p(\cdot)}\Big|\mathcal {F}_n\Big),\end{equation}
then
\begin{equation}\label{f2}
\mathbb{P}(\sup_n|f_n|>\lambda)\leq{C_{p(\cdot)}
\int_{\Omega}\Bigg(\frac{|f_\infty|}{\lambda}\Bigg)^{p(\cdot)}d{\mathbb{P}}},\quad
\forall \lambda >0.\end{equation}
And if $1<p_-\leq p_+<\infty$ and
 $p(\cdot)$ is $\mathcal
{F}_n$-measurable for all $n\geq0$, then
\begin{equation}\label{f3}
\|\sup_n|f_n|\|_{p(\cdot)} \leq C_{p(\cdot)} \|f\|_{p(\cdot)}.
\end{equation}
Obviously, the condition that $p(\cdot)$ is $\mathcal
{F}_n$-measurable for all $n\geq0$ is quite strict. In 2013, Nakai and Sadasue
\cite{NS} pointed out that there exists a variable exponent $p(\cdot)
$ such that $p(\cdot) $ is not $\mathcal {F}_0$-measurable, but \eqref{f3} still holds. In this paper, we obtain the weak type
inequality \eqref{f2} without condition \eqref{f1}. Unfortunately we cannot obtain \eqref{f3} directly by the
weak type inequality as the classical case. This is
because that the space $L^{p(\cdot)}$ is no longer a rearrangement
invariant space, and the formula
$$ \int_\Omega |f(x)|^p d\mathbb{P}= p \int_0^\infty t^{p-1}\mathbb{P}( x\in \Omega : |f(x)|>t )dt$$
has no variable exponent analogue (see \cite{DH}). In order to describe the strong type Doob maximal inequality, we find the following condition without
metric characterization of $p(x)$
to replace log-H\"{o}lder continuity in some sense.  That is, there exists an absolute constant $K_{p(\cdot)}\geq1$ depending only on $p(\cdot)$ such that
\begin{equation}\label{65} \mathbb{P}(A)^{p_-(A)-p_+(A)} \leq K_{p(\cdot)}, \quad\forall A\in \mathcal{F}. \end{equation} We often denote $K_{p(\cdot)}$ simply by $K$ if there is nothing confused.  Under the condition of \eqref{65},
 we prove \eqref{f3} is true for any martingale with respect to the atom $\sigma$-algebra
 filtration. It should be mentioned that the condition \eqref{65} is not true usually
(even in Euclidean space); however if the exponent $p(x)$ has a nice
uniform continuity with respect to Euclidean distance, then \eqref{65}
holds. We refer to Lemma 3.2 in \cite{diening} for this fact.

The second aim of this paper is the atomic characterization of variable Hardy martingale spaces. Our result can be regarded as the probability version of \cite{cruz2,NS2}; we do not use the log-H\"{o}lder continuity of $p(x)$ and it seems that our proofs are simpler because of stopping time techniques. Let $\mathcal{T}$ be the set of all stopping times with respect to $\{\mathcal{F}_n\}_{n \geq 0}$. For a martingale
$f=(f_n)_{n\geq0}$ and $\tau \in \mathcal{T}$, we denote the stopped martingale by $f^\tau=(f_n^\tau)_{n\geq0}=(f_{n\wedge\tau})_{n\geq0}$.

\begin{definition} \label{definition11}Given $p(\cdot) \in \mathcal{P}$. A
measurable function $a$ is called a $(1,p(\cdot),\infty)$-atom if
there exists a stopping time $\tau\in\mathcal{T}$ such that
\begin{enumerate}
\item $\E(a|\mathcal {F}_n)=0$,  $\forall\; n\leq \tau$,

\item $\|s(a)\|_\infty \leq \left\|\chi_{\{\tau<\infty\}}\right\|^{-1}_{p(\cdot)}.$
\end{enumerate}
\end{definition}

Denote by $A(s,p(\cdot),\infty)$ be the set of all sequences of pair
$(\mu_k, a^k, \tau_k)$, where $\mu_k$ are
nonnegative numbers, $a^k$ are $(1,p(\cdot),\infty)$-atoms satisfying $(1)$, $(2)$.

In the sequel we always denote $\underline{p}=\min\{p_-,1\}$.

\begin{definition}\label{definition12} Given $p(\cdot) \in \mathcal{P}$.
Let us denote by $H_{p(\cdot)}^{s, at}$ the space of those
martingales for which there exist a sequence $(a^k)_{k\in
\mathbb{Z}}$ of $(1,p(\cdot),\infty)$-atoms and a sequence
$(\mu_k)_{k\in \mathbb{Z}}$ of nonnegative real numbers such
that
\begin{equation}\label{11}
f=\sum_{k\in\mathbb{Z}}\mu_k a^k, \quad a.e. ,
\end{equation}
and
$$\Bigg\|\Bigg\{\sum_{k\in\mathbb{Z}}\Bigg(\frac{\mu_k\chi_{\{\tau_k<\infty\}}}
{\|\chi_{\{\tau_k<\infty\}}\|_{p(\cdot)}}\Bigg)^{\underline{p}}\Bigg\}^{\frac{1}{\underline{p}}}\Bigg\|_{p(\cdot)} < \infty.$$
Let
$$\mathcal{A}(\{\mu_k\},\{a^k\},\{\tau_k\})\equiv\Bigg\|\Bigg\{\sum_{k\in\mathbb{Z}}\Bigg(\frac{\mu_k\chi_{\{\tau_k<\infty\}}}
{\|\chi_{\{\tau_k<\infty\}}\|_{p(\cdot)}}\Bigg)^{\underline{p}}\Bigg\}^{\frac{1}{\underline{p}}}\Bigg\|_{p(\cdot)}.$$
We define
$$\|f\|_{H_{p(\cdot)}^{s, at}}=\inf\mathcal{A}(\{\mu_k\},\{a^k\},\{\tau_k\}),~~~~(\mu_k, a^k, \tau_k)\in A(s,p(\cdot),\infty),$$
where the infimum is taken over all decompositions of the form \eqref{11}. \end{definition}
In Section 4, we prove that  $$H_{p(\cdot)}^s
=H_{p(\cdot)}^{s, at},\quad\quad  p(\cdot)\in \mathcal{P},$$
with equivalent quasi-norms, see Section 2 for the notation $H_{p(\cdot)}^s$. We give some applications of atomic decomposition in
Section 5.
Recall that the Lipschitz space $\Lambda_q(\alpha) (\alpha\geq0, q\geq1),$ is defined as
 the space of all functions $f\in L^q$ for which
$$\|f\|_{\Lambda_q(\alpha)}= \sup_{\tau}|\{ \tau< \infty\}|^ {-\frac{1}{q}-\alpha} \|f-f^\tau\|_q<\infty.$$
It was proved by Weisz in \cite{FW2} that the dual space of
$H_p^s(0<p\leq1)$ is equivalent to
$\Lambda_2(\alpha)(\alpha=1/p-1)$.  The new Lipschitz
space $\Lambda_q(\alpha(\cdot))$ is introduced in Section 5. Let $p(\cdot)$ satisfy \eqref{65}.  We obtain that
$$\Big(H_{p(\cdot)}^s\Big)^*=\Lambda_2\big(\alpha(\cdot)\big),\quad 0<p_- \leq
p_+\leq1,$$ where $\alpha(\cdot)=1/p(\cdot)-1$.

 Finally we get the John-Nirenberg inequality in the frame of variable exponents. If $p(\cdot)$ satisfies \eqref{65}, then
$$\|f \|_{BMO_1}\lesssim \|f \|_{BMO_{p(\cdot)}} \lesssim \|f \|_{BMO_1},\quad1\leq p_- \leq p_+<\infty ,$$
which can be regarded as the probability versions of Theorem 1.2 or Theorem 5.1 in \cite{IST}. See Section 5 for the definition of   $BMO_{p(\cdot)}$.
Furthermore, we also obtain the exponential integrability form of the John-Nirenberg inequality for $BMO_{p(\cdot)},$ which is the probability analogue of Theorem 3.2 in \cite{HO}.
We note that the generalized John-Nirenberg inequalities were proved in the frame of rearrangement invariant spaces in \cite{ywj}, but the
variable $L^{p(\cdot)}$ spaces are not rearrangement invariant spaces except that $p(\cdot)$ is a constant. Again, the condition \eqref{65} plays an important role in the above results, which admits us to estimate the $p(\cdot)$-norm of characterization function and  makes inverse H\"{o}lder's
inequalities available.

Throughout this paper, $\mathbb{Z}$, $\mathbb{N}$ and $\mathbb{C}$
denote the integer set, nonnegative integer set and complex numbers
set, respectively. We denote by $C$ the absolute positive constant,
which can vary from line to line, and denote by $C_{p(\cdot)}$ the
constant dependently only on $p(\cdot).$ The symbol $A\lesssim B$
stands for the inequality $A \leq C B$ or $A \leq C_{p(\cdot)} B$.  If we
write $A\thickapprox B$, then it stands for $A\lesssim B\lesssim A$.

\section{Preliminaries and lemmas}

In this section, we give some preliminaries necessary to the whole
paper. Given $p(\cdot) \in \mathcal{P}$, we always assume that
$0<p_-\leq p_+<\infty$ if no special statement. The space $L^{p(\cdot)}=L^{p(\cdot)}(\Omega)$ is the
collection of all measurable functions $f$ defined on $(\Omega,\mathcal {F},\mathbb{P})$ such that for some $\lambda>0$,
$$\rho({f}\left/{\lambda}\right)=\int_{\Omega}\Bigg(\frac{|f(x)|}{\lambda}\Bigg)^{p(x)}d\mathbb{P}<\infty.$$
This becomes a quasi-Banach function space when it is equipped with the
quasi-norm
$$\|f\|_{p(\cdot)}\equiv\inf\{\lambda>0:\rho({f}\left/{\lambda}\right)\leq1\}.$$

The following facts are well known; see for example \cite{NS2}.
\begin{enumerate}
\item (Positivity) $\|f\|_{p(\cdot)} \geq 0$;  $\|f\|_{p(\cdot)}=0
\Leftrightarrow f\equiv 0$.

\item (Homogeneity) $\|cf\|_{p(\cdot)}=|c| \cdot \|f\|_{p(\cdot)}$ for
$c\in \mathbb{C}$.

\item (The $\underline{p}$-triangle inequality) $\|f+g\|_{p(\cdot)}^{\underline{p}} \leq
\|f\|_{p(\cdot)}^{\underline{p}} +\|g\|_{p(\cdot)}^{\underline{p}}$.
\end{enumerate}
For $p(\cdot) \in \mathcal{P}$ and $p_-> 1$, we define the conjugate exponent
$p'(\cdot)$ by the equation
$$\frac{1}{p(x)}+\frac{1}{p'(x)}=1.$$

 We collect some useful lemmas as follows, which will be used in the
paper.

\begin{lemma} \label{lemma21} (see \cite{cruz2}, page 5) Let $p(\cdot) \in
\mathcal{P}$, and $p_-\geq 1$ then for all $r>0$, we
have
$$\left\||f|^r\right\|_{p(\cdot)}=\|f\|_{rp(\cdot)}^r.$$
\end{lemma}

\begin{lemma} \label{lemma22} (see \cite{CF}, page 24) Given $p(\cdot) \in
\mathcal{P}$, then for all $f\in L^{p(\cdot)}$ and
$\|f\|_{p(\cdot)}\neq 0$, we have
$$\int_\Omega \left| \frac{f(x)}{\|f\|_{p(\cdot)}} \right|^{p(x)} d\mathbb{P}=1.$$
\end{lemma}

\begin{lemma} \label{lemma23} (see \cite{FZ}, Theorem 1.3 or \cite{CF}, page 22 ) Given $p(\cdot) \in
\mathcal{P}$ and $f\in L^{p(\cdot)}$, then we have
\begin{enumerate}
\item $\|f\|_{p(\cdot)}<1(=1,>1)$ if and only if $\rho(f)<1(=1,>1)$;

\item If $\|f\|_{p(\cdot)}>1$, then $\rho(f)^{1/p_+} \leq
\|f\|_{p(\cdot)} \leq \rho(f)^{1/p_-}$;

\item If $0<\|f\|_{p(\cdot)}\leq 1$, then $\rho(f)^{1/p_-} \leq
\|f\|_{p(\cdot)} \leq \rho(f)^{1/p_+}$.
\end{enumerate}
\end{lemma}

\begin{lemma} \label{lemma24} (see \cite{CF}, H\"{o}lder's inequality) Given
$p(\cdot),\ q(\cdot),\ r(\cdot) \in \mathcal{P}$, such that
$$\frac{1}{p(x)} = \frac{1}{q(x)} + \frac{1}{r(x)}.$$
Then there exists a constant $C_{p(\cdot)}$ such that for all $f \in
L^{q(\cdot)}$, $g\in L^{r(\cdot)}$, and $\ fg \in L^{p(\cdot)}$
$$\|fg\|_{p(\cdot)} \leq C_{p(\cdot)} \|f\|_{q(\cdot)}\|g\|_{r(\cdot)}.$$
\end{lemma}

Now we introduce some standard notations from martingale theory. We refer to \cite{Gar,LRL, FW1} for the classical
martingale space theory. Let
$(\Omega,\mathcal {F},\mathbb{P})$ be a complete probability space.
Recall that the conditional expectation operator relative to
${\mathcal{F}_n}$ is denoted by $\E_{\mathcal{F}_n}$, i.e.
$\E(f|{\mathcal{F}_n})=\E_{\mathcal{F}_n}(f)$. A sequence of
measurable functions $f=(f_n)_{n\geq0}\subset L^1(\Omega)$ is called
a martingale with respect to $(\mathcal {F}_n)$ if $\E_{\mathcal
{F}_n}(f_{n+1})=f_n$ for every $n\geq0.$ If in addition $f_n\in
L^{p(\cdot)}$, $f$ is called an $L^{p(\cdot)}$-martingale with
respect to $(\mathcal {F}_n)$. In this case we set
$$\|f\|_{p(\cdot)}=\sup_{n\geq0}\|f_n\|_{p(\cdot)}.$$ If
$\|f\|_{p(\cdot)}<\infty$, $f$ is called a bounded
$L^{p(\cdot)}$-martingale and denoted by $f\in L^{p(\cdot)}$. For a
martingale relative to $(\Omega,\mathcal {F},\mathbb{P};(\mathcal
{F}_n)_{n\geq 0})$, define the maximal function and the conditional square
function of $f$ respectively as follows $(f_{-1}=f_0)$,
$$M_mf=\sup_{n\leq m}{|f_n|}, ~~~~Mf=\sup_{n\geq1} |f_n|,$$
$$s_m(f)=\left(\sum\limits_{n=0}^{m}\E_{\mathcal{F}_{n-1}}|df_n|^2\right)^{\frac{1}{2}},
~~~~s(f)=\left(\sum\limits_{n=0}^{\infty}\E_{\mathcal{F}_{n-1}}|df_n|^2\right)^{\frac{1}{2}}.$$
Then we define the variable exponent martingale Hardy spaces
analogous to
classical martingale Hardy spaces as follows
$$H^*_{p(\cdot)}=\Big\{f=(f_n)_{n\geq0}: Mf\in L^{p(\cdot)}\Big\},~~~~~~\|f\|_{H^*_{p(\cdot)}}={\|Mf\|}_{p(\cdot)}.$$
$$H^s_{p(\cdot)}=\Big\{f=(f_n)_{n\geq0}: s(f)\in L^{p(\cdot)}\Big\},~~~~~~\|f\|_{H^s_{p(\cdot)}}={\|s(f)\|}_{p(\cdot)}.$$

\section{Doob's maximal inequalities}

In this section we first prove the weak type inequality \eqref{f2} without the condition \eqref{f1}. We begin with the  following lemma.

\begin{lemma}\label{lemma31} Given $p(\cdot)\in \mathcal{P}$ and
$1\leq p_-\leq p_+ < \infty$. Let $f=(f_n)_{0\leq n \leq \infty}$ be a
$L^{p(\cdot)}$-martingale. Suppose that for any stopping time $\tau$
$$\mathbb{P}(\tau<\infty)<\int_{\{\tau<\infty\}}\frac{|f_{\infty}|}{\lambda}d\mathbb{P}, \quad \forall\lambda>0.$$
Then there exists a constant $C_{p(\cdot)}$ such that
$$
\mathbb{P}(\tau<\infty)\leq C_{p(\cdot)}
\int_{\{\tau<\infty\}}\Bigg(\frac{|f_{\infty}|}{\lambda}\Bigg)^{p(x)}d{\mathbb{P}}, \quad \forall\lambda>0.$$ \end{lemma}

\begin{proof} We choose a sequence of simple functions $\{s_n\}_{n\geq1}$
such that $p_+(\{\tau<\infty\}) \geq s_n\geq p_-(\{\tau<\infty\}) $ for any $n$ and the
sequence $\{s_n\}_{n\geq1}$ increases monotonically to $p(x)$ on
${\{\tau<\infty\}}$. Then for each $n$
$$s_n(x)=\sum\limits_{j=1}^{k_n}\alpha_{n,j}\chi_{A_{n,j}}(x),$$
where the sets $\{A_{n,j}\}$ are disjoint and
$\bigcup_j{A_{n,j}}=\{\tau<\infty\}$.

By H\"{o}lder's inequality and Young's inequality we have
\begin{eqnarray*}
\int_{A_{n,j}}\frac{|f_\infty(x)|}{\lambda}d{\mathbb{P}} &\leq&
\left(\int_{A_{n,j}}\Bigg(\frac{|f_\infty(x)|}{\lambda}\Bigg)^{\alpha_{n,j}}d{\mathbb{P}}\right)^{\frac{1}{\alpha_{n,j}}}
\mathbb{P}(A_{n,j})^{\frac{1}{\alpha_{n,j}'}}
\\&\leq& \frac{1}{\alpha_{n,j}}\int_{A_{n,j}}\Bigg(\frac{|f_\infty(x)|}{\lambda}\Bigg)^{\alpha_{n,j}}d{\mathbb{P}}+{\frac{\mathbb{P}({A_{n,j})}}{\alpha_{n,j}'}}
\\&\leq&\frac{1}{p_-(\{\tau<\infty\})}\int_{A_{n,j}}\Big(\frac{|f_\infty(x)|}{\lambda}\Big)^{s_n(x)}d{\mathbb{P}}+ {\frac{\mathbb{P}({A_{n,j})}}{\big(p_+\big(\{\tau<\infty\}\big)\big)'}}.
\end{eqnarray*}
Adding the above inequalities with $j$ from $1$ to $k_n$, we have
$$\int_{\{\tau<\infty\}}\frac{|f_\infty(x)|}{\lambda}d\mathbb{P}\leq \frac{1}{p_-\big(\{\tau<\infty\}\big)}\int_{\{\tau<\infty\}}\Big(\frac{|f_\infty(x)|}{\lambda}\Big)^{s_n(x)}d{\mathbb{P}}+ {\frac{\mathbb{P}({\tau<\infty)}}{\big(p_+\big(\{\tau<\infty\}\big)\big)'}}.$$
This inequality holds for all $n$, hence the monotone convergence
theorem implies that
\begin{multline}\mathbb{P}(\tau<\infty)< \int_{\{\tau<\infty\}} \frac{|f_\infty|}{\lambda}d\mathbb{P}\\ \leq \frac{1}{p_-\big(\{\tau<\infty\}\big)} \int_{\{\tau<\infty\}}\Big(\frac{|f_\infty(x)|}{\lambda}\Big)^{p(x)}d{\mathbb{P}}+ {\frac{\mathbb{P}({\tau<\infty)}}{\big(p_+\big(\{\tau<\infty\}\big)\big)'}}.
\end{multline}
Since $p_+<\infty$, then $\big(p_+(\{\tau<\infty\})\big)'>1$. It
follows that
$$\mathbb{P}(\tau<\infty)\Big(1-{\frac{1}{(p_+(\{\tau<\infty\}))'}}\Big) \leq \frac{1}{p_-\big(\{\tau<\infty\}\big)} \int_{\{\tau<\infty\}}\Bigg(\frac{|f_\infty(x)|}{\lambda}\Bigg)^{p(x)}d{\mathbb{P}}.$$
Therefore, by a simple calculation, we have
$$\mathbb{P}(\tau<\infty)\leq C_{p(\cdot)} \int_{\{\tau<\infty\}}\Bigg(\frac{|f_\infty(x)|}{\lambda}\Bigg)^{p(x)}d{\mathbb{P}}.$$
\end{proof}

The following theorem corresponds to Proposition 4 in \cite{HA}.

\begin{theorem} Given $p(\cdot)\in \mathcal{P}$ and
$1\leq p_-\leq p_+ < \infty$. Suppose that $f=(f_n)_{0\leq n\leq\infty}$ is a
bounded $L^{p(\cdot)}$-martingale, then
$$\mathbb{P}(Mf>\lambda)\leq {C_{p(\cdot)} \int_{\Omega}\Bigg(\frac{|f_\infty(x)|}{\lambda}\Bigg)^{p(x)}d{\mathbb{P}}},\quad \forall \lambda>0.$$ \end{theorem}

\begin{proof}For any $\lambda>0$, we define a stopping time
$\tau=\inf\{n>0:|f_n|>\lambda\}$ ( with the convention that
$\inf\varnothing=\infty$). It is obvious that
$$\{Mf>\lambda\}=\{\tau<\infty\},$$
and
$$\{\tau<\infty\}\subset\{|f_{\tau}|>\lambda\}.$$
Note that
${\E_{\mathcal{F}_{\tau}}\Big({\frac{|f_\infty|}{\lambda}}\Big)}>1$ a.e. on
the set $\{\tau<\infty\}$. We get that
\begin{eqnarray*}
\mathbb{P}(\tau<\infty)&=&\int_{\{\tau<\infty\}}1d{\mathbb{P}}
\leq\int_{\{\tau<\infty\}}{\E_{\mathcal{F}_{\tau}}\Bigg({\frac{|f_\infty|}{\lambda}}\Bigg)}d{\mathbb{P}}
\\&=&\int_{\{\tau<\infty\}}{\frac{|f_\infty(x)|}{\lambda}}d{\mathbb{P}}.
\end{eqnarray*}
It  follows immediately from  Lemma \ref{lemma31} that
\begin{eqnarray*}
\mathbb{P}(Mf>\lambda)&=&\mathbb{P}(\tau<\infty)
\leq C_{p(\cdot)}\int_{\{\tau<\infty\}}\Bigg(\frac{|f_\infty(x)|}{\lambda}\Bigg)^{p(x)}d{\mathbb{P}}
\\&\leq& C_{p(\cdot)}\int_{\Omega}\Bigg(\frac{|f_\infty(x)|}{\lambda}\Bigg)^{p(x)}d{\mathbb{P}}.
\end{eqnarray*}
The proof is complete.
\end{proof}

\begin{lemma} Given $p(\cdot)\in \mathcal{P}$. Then
$$\left(\sup_{n\geq0}|f_n|\right)^{p(\cdot)}=\sup_{n\geq0}\left(|f_n|^{p(\cdot)}\right).$$ \end{lemma}
This lemma is very obvious, however it will be used frequently below.

We now turn to consider the strong type inequality \eqref{f3}.
Let $(\Omega,\mathcal{F},\mathbb{P})$ be a probability space. Let
$$\mathcal{D}_n=\{A_j^n\}_{j\geq1},~~~~~~~ \mbox{for each}~n\geq0,$$
be decompositions of $\Omega$ such that $(\mathcal{B}_n)_{n\geq0}=\left(\sigma(\mathcal{D}_n)\right)_{n\geq0}$ is increasing and  $\mathcal
 {F}=\mathcal {\sigma }\left(\bigcup_{n\geq 0}\mathcal {B}_n \right)$.
It is clear that
$$\E_{\mathcal{B}_n}(f)=\sum_{j=1}^\infty \left( \frac{1}{\mathbb{P}\left( A_j^n\right)} \int_{A_j^n}f(x) d\mathbb{P} \right) \chi_{A_j^n}.$$
Then
\begin{eqnarray} \label{32}
\int_{\Omega}(Mf)^{p(x)} d\mathbb{P} &\leq&  \int_{\Omega} \sup_n \left\{ \sum_{j=1}^\infty \Big( \frac{1}{\mathbb{P}\Big( A_j^n\Big)} \int_{A_j^n}|f(x)| d\mathbb{P} \Big) \chi_{A_j^n} \right\}^{p(x)} d\mathbb{P} \nonumber
\\&=& \int_{\Omega}  \left\{ \sup_{n}\sum_{j=1}^\infty \Big( \frac{1}{\mathbb{P}\Big( A_j^n\Big)} \int_{A_j^n}|f(x)| d\mathbb{P} \Big)^{\frac{p(x)}{p_-}} \chi_{A_j^n}  \right\}^{p_-} d\mathbb{P}.
\end{eqnarray}

\begin{lemma} \label{lemma34} Let $p(\cdot)\in \mathcal{P}$, $1< p_-\leq p_+ < \infty$ and satisfy \eqref{65}.
Suppose that $f\in L^{p(\cdot)}$ and $\left\|f\right\|_{p(\cdot)}\leq 1/2$. Then for all measurable sets $B$,
$$\left(\frac{1}{\mathbb{P}(B)} \int_B |f(x)| d\mathbb{P}\right)^{\frac{p(x)}{p_-}} \leq
K\left(\frac{1}{\mathbb{P}(B)} \int_B |f(x)|^{\frac{p(x)}{p_-}} d\mathbb{P}+1\right).$$ \end{lemma}
\begin{proof}Let $q(x)=p(x)/{p_-}$, then for any $x\in B$,
$$q(x)\leq p(x),\quad\mbox{and}\quad 1\leq q_-(B)\leq p(x).$$
Let $f|_B(x)=f(x),\,x\in B.$ Then $\|f|_B\|_{p(\cdot)}\leq 1/2.$ Let $g=f/{\|f|_B\|_{p(\cdot)}}$. It follows from Lemma \ref{lemma22} that
\begin{eqnarray*}
\int_B \Big|\frac{f(x)}{\|f|_B\|_{p(\cdot)}} \Big|^{q_-(B)} d\mathbb{P}
&=&
\int_{B\cap \{|g|\geq1\}} |g(x)|^{q_-(B)} d\mathbb{P}+\int_{B\cap \{|g|<1\}} |g(x)|^{q_-(B)} d\mathbb{P}
\\&\leq& 1+ \mathbb{P}(\Omega).
\end{eqnarray*}
Then
\begin{eqnarray*}
\left\|f|_B\right\|_{q_-(B)}
&\leq&(1+\mathbb{P}(\Omega))^{\frac{1}{q_-(B)}}\|f|_B\|_{p(\cdot)}
\leq(1+ \mathbb{P}(\Omega))\|f\|_{p(\cdot)}\leq1.
\end{eqnarray*}
Using H\"{o}lder's inequality and \eqref{65}, we find that
\begin{eqnarray*}
\left(\frac{1}{\mathbb{P}(B)} \int_B |f(y)| d\mathbb{P}\right)^{q(x)}
&\leq&
\left(\frac{1}{\mathbb{P}(B)} \int_B |f(y)|^{q_-(B)} d\mathbb{P}\right)^{\frac{q(x)}{{q_-(B)} }}
\\&=&
\mathbb{P}(B)^{-\frac{q(x)}{{q_-(B)} }}\left\|f|_B\right\|_{q_-(B)}^{q(x)}
\\&\leq&\mathbb{P}(B)^{-\frac{q(x)}{{q_-(B)} }}\left\|f|_B\right\|_{q_-(B)}^{q_-(B)}
\\&=&\mathbb{P}(B)^{-\frac{q(x)-q_-(B)}{{q_-(B)}}}\frac{1}{\mathbb{P}(B)} \int_B |f(x)|^{q_-(B)}d\mathbb{P}
\\&\leq&\mathbb{P}(B)^{\frac{q_-(B)-q_+(B)}{{q_-(B)}}}\frac{1}{\mathbb{P}(B)} \int_B |f(y)|^{q_-(B)}d\mathbb{P}
\\&=&\mathbb{P}(B)^{\frac{p_-(B)-p_+(B)}{{p_-(B)}}}\frac{1}{\mathbb{P}(B)} \int_B |f(y)|^{q_-(B)}d\mathbb{P}
\\&\leq&K^{\frac{1}{p_-(B)} }\left( \frac{1}{\mathbb{P}(B)} \int_B\Big( |f(y)|^{q(y)}+1\Big) d\mathbb{P}\right)
\\&\leq& K\left( \frac{1}{\mathbb{P}(B)} \int_B\Big( |f(y)|^{q(y)}+1\Big) d\mathbb{P}\right).
\end{eqnarray*}
\end{proof}

\begin{theorem} \label{maximal boundedness}Let
$\mathcal{D}_n=\{A_j^n\}_{j\geq1}$, for each $n\geq0$,
be decompositions of $\Omega$ such that $(\mathcal{B}_n)_{n\geq0}=\left(\sigma(\mathcal{D}_n)\right)_{n\geq0}$ is increasing and  $\mathcal
 {F}=\mathcal {\sigma }\left(\bigcup_{n\geq 0}\mathcal {B}_n \right)$. Let $p(\cdot)$ satisfy \eqref{65} and $1< p_-\leq p_+ < \infty$.
Then for any martingale $f\in L^{p(\cdot)}$ with respect to $(\mathcal{B}_n)_{n\geq0},$
$$\|\sup_n|f_n|\|_{p(\cdot)} \leq C_{p(\cdot)} \|f\|_{p(\cdot)}.$$ \end{theorem}
\begin{proof} We assume that $\left\|f\right\|_{p(\cdot)}\leq 1/2$ by homogeneity and let $q(x)=p(x)/{p_-}$. Then by Lemma \ref{lemma34} and the classical Doob maximal inequality
\begin{align*}
 & \mbox{\quad\quad}\int_{\Omega}  \left\{ \sup_{n} \sum_{j=1}^\infty \left( \frac{1}{\mathbb{P}\left( A_j^n\right)} \int_{A_j^n}|f(x)| d\mathbb{P} \right)^{\frac{p(x)}{p_-}} \chi_{A_j^n}  \right\}^{p_-} d\mathbb{P} \\
 &\leq \int_{\Omega}  \left\{\sup_{n} \sum_{j=1}^\infty K\left( \frac{1}{\mathbb{P}\left( A_j^n\right)} \int_{A_j^n}\left(|f(x)|^{\frac{p(x)}{p_-}} +1\right)d\mathbb{P} \right) \chi_{A_j^n}  \right\}^{p_-} d\mathbb{P} \\
 &=K^{p_-}\left\| \sup_n \E_{\mathcal{B}_n} \left( |f|^{q(\cdot)}+1\right) \right\|_{p_-}^{p_-} \\
 &\leq C_{p_-}K^{p_-} \left \||f|^{q(\cdot)}+1\right\|_{p_-}^{p_-}  \leq C
\end{align*}
By \eqref{32}, we have $\int_\Omega (Mf)^{p(x)}d\mathbb{P}\leq C$. Now the proof is complete. \end{proof}

\begin{remark}
(1) We point out that there is a non-log-H\"{o}lder continuous function $p(\cdot)$ for which the maximal operator is bounded on the corresponding Lebesgue spaces $L_{p(\cdot)}(\mathbb{R}^n)$; see \cite{Nekvinda}.

(2) Note that condition \eqref{65} could not cover the example given by Nakai and Sadasue (p.2169, \cite{NS}). Indeed, we can verify a special case of their example.
Let $((0,1],\Sigma, \mu)$ be a probability space such that $\mu$ is the Lebesgue measure and subalgebras $\{\Sigma_n\}_{n\geq0}$ generated as follows
$$\Sigma_n=\sigma\mbox{-algebra generated by atoms} \quad \big(\frac{j}{2^n},\frac{j+1}{2^n}\big],j=0,\cdots,2^n-1. $$
For $n\geq 0$ we set $B_n=\big(0,\frac{1}{2^n}\big]$, then
$$(0,1]=B_0\supset B_1 \supset \cdots \supset B_n \cdots,$$
 and let
 $$g(x)=\sin ( h(x)),\quad h(x)=\sum_{n=1}^\infty \frac{1}{\ln(2^ne)}(2\chi_{B_n}-\chi_{B_{n-1}}).$$
Denote $h_m:=\sum_{n=1}^{m} \frac{1}{\ln(2^ne)}-\frac{1}{\ln(2^{m+1}e)}$, $m\geq1$.
It is easy to check that
\begin{equation} \label{hm} h_{m}\rightarrow \infty \quad \mbox{as} \quad m\rightarrow \infty.\end{equation}
 Also, we have
\begin{equation}\label{dyadic example} 0<h_{m+1}-h_{m}\leq \frac{2}{(m+1)\ln2}< \frac{2\pi}{3},\quad m\geq 1.\end{equation}

 Given $N$, we shall show that there exists $y\in B_N $ such that $1\geq g(y)\geq 1/2$.
Choose the smallest integer $k$ so that $h_N < 2k\pi+\frac{\pi}{6}$. Then from \eqref{hm} and \eqref{dyadic example}, it follows that there exists $j> N$ satisfying $h_j\in  (2k\pi+\frac{\pi}{6},2k\pi+\frac{5\pi}{6})$.
This means  for any  $y\in B_j\setminus B_{j+1}\subset B_N$, we have $ 1\geq g(y)\geq 1/2$.
 Similarly, there exists $z\in B_N$ such that $-1\leq g(z)\leq 0$. Now we obtain
$$\mu(B_N)^{g_-(B_N)-g_+(B_N)}=(2^N)^{g_+(B_N)-g_-(B_N)}\geq (2^N)^{g(y)-g(z)} \geq (2^N)^{1/2},$$
which implies that $g(\cdot)$ does not satisfy condition \eqref{65}.
\end{remark}

In the time of this writing, we do not know if the condition \eqref{65} is sufficient for the Doob maximal inequality in general probability spaces.

\begin{problem} \label{problem36}Let $p(\cdot)$ satisfy \eqref{65} with $1< p_-\leq p_+ < \infty$.
Then for any martingale $f\in L^{p(\cdot)}$ with respect to $(\mathcal{F}_n)_{n\geq0},$
$$\|\sup_n|f_n|\|_{p(\cdot)} \leq C_{p(\cdot)} \|f\|_{p(\cdot)} \,\,?$$ \end{problem}

\begin{remark} It is well known that $\big|\mathbb{E}_{\mathcal{F}_n}(f)\big|^p\leq \mathbb{E}_{\mathcal{F}_n}(|f|^p)$ for $1\leq p<\infty.$
However, it is easy to give inverse examples to show that one can never expect a variable exponent version, namely,
\begin{eqnarray}\label{33}
\big|\mathbb{E}_{\mathcal{F}_n}(f)\big|^{p(\cdot)}\leq C_{p(\cdot)}\mathbb{E}_{\mathcal{F}_n}(|f|^{p(\cdot)}),\quad 1\leq p(\cdot)<\infty.
\end{eqnarray}
Hence the main difficulty to deal with Problem \ref{problem36} is how to overcome or avoid the use of the inequality \eqref{33}. \end{remark}

\section{Atomic characterization of variable Hardy martingale space }

In this section we construct the atomic decoposition of martingale Hardy space with variable exponents.
Here we use Definitions \ref{definition11} and \ref{definition12}.

\begin{proposition} \label{proposition41}Given $p(\cdot) \in \mathcal{P}$. Let $f\in
H_{p(\cdot)}^{s,at}$, i.e., $f=\sum\mu_ka^k.$
\begin{enumerate}
\item We have
$$\left(\sum_{k\in\mathbb{Z}} \mu_k^{p_+} \right)^{\frac{1}{p_+}} \leq \mathcal{A}(\{\mu_k\},\{a^k\},\{\tau_k\}).$$

\item If $p_+ \leq 1$, then $$\sum_{k\in\mathbb{Z}}\mu_k \leq
\mathcal{A}(\{\mu_k\},\{a^k\},\{\tau_k\}).$$

\item For any $k\in \mathbb{Z}$ we have
$$\|a^k\|_{H_{p(\cdot)}^{s, at}}\leq 1.$$
\end{enumerate}
\end{proposition}

\begin{proof}  (1) The convexity implies
that
\begin{eqnarray*}
\int_\Omega \left( \sum_{k\in\mathbb{Z}} \left( \frac{\mu_k
\chi_{\{\tau_k<\infty\}}}{\lambda
\|\chi_{\{\tau_k<\infty\}}\|_{p(\cdot)}} \right)^{\underline{p}}
\right)^{\frac{p(x)}{\underline{p}}} d\mathbb{P} &\geq& \int_\Omega
\sum_{k\in\mathbb{Z}} \left( \frac{\mu_k
\chi_{\{\tau_k<\infty\}}}{\lambda
\|\chi_{\{\tau_k<\infty\}}\|_{p(\cdot)}} \right)^{p(x)}
d\mathbb{P}
\\&=&
\sum_{k\in\mathbb{Z}} \int_{\{\tau_k<\infty\}}  \left( \frac{\mu_k
}{\lambda \|\chi_{\{\tau_k<\infty\}}\|_{p(\cdot)}} \right)^{p(x)}
d\mathbb{P}
\end{eqnarray*}
Now if we set $\lambda=\left(\sum_{k\in\mathbb{Z}} \mu_k^{p_+}
\right)^{\frac{1}{p_+}}$, and then we obtain
$$\int_\Omega \Big( \sum_{k\in\mathbb{Z}} \Big( \frac{\mu_k \chi_{\{\tau_k<\infty\}}}{\lambda \|\chi_{\{\tau_k<\infty\}}\|_{p(\cdot)}} \Big)^{\underline{p}}  \Big)^{\frac{p(x)}{\underline{p}}} d\mathbb{P} \geq
\sum_{k \in \mathbb{Z}} \left(\frac{\mu_k}{\lambda}\right)^{p_+}
\int_\Omega \left( \frac{ \chi_{\{\tau_k < \infty\}}} {
\|\chi_{\{\tau_k < \infty\}}\|_{p(\cdot)}} \right)^{p(x)}
d\mathbb{P} =1.$$ By the definition of
$\mathcal{A}(\{\mu_k\},\{a^k\},\{\tau_k\})$, we get the desired result.

(2) and (3) are obvious.
\end{proof}

\begin{theorem} \label{theorem42} Let $p(\cdot)\in\mathcal{P}$. If the martingale $f\in H_{p(\cdot)}^s$, then
there exist a sequence $(a^k)_{k\in
\mathbb{Z}}$ of $(1,p(\cdot),\infty)$-atoms and a sequence
$(\mu_k)_{k\in \mathbb{Z}}$ of nonnegative real numbers such that for all $n\geq0$,
\begin{equation}\label{41}
\sum_{k\in \mathbb{Z}} \mu_k \E_{\mathcal{F}_n}a^k=f_n, \quad\mbox{a.e}
\end{equation}
and
$$\mathcal{A}(\{\mu_k\},\{a^k\},\{\tau_k\}) \lesssim {H_{p(\cdot)}^s}.$$
Moreover the sum $\sum_{k\in \mathbb{Z}} \mu_k a^k$ converges to $f$ in ${H_{p(\cdot)}^s}$.
Conversely, if the martingale $f$ has a decomposition of \eqref{41}, then
$$\|f\|_{H_{p(\cdot)}^s} \lesssim \inf \mathcal{A}(\{\mu_k\},\{a^k\},\{\tau_k\}),$$ where the infimum is taken over all the decompositions of the form \eqref{41}. \end{theorem}

\begin{proof}
Assume that $f \in H_{p(\cdot)}^s$. Let us consider the following
stopping times for all $k \in \mathbb{Z}$
$$\tau_k=\inf\{n\in\mathbb{N}: s_{n+1}(f)>2^k \}.$$
The sequence of these stopping times is obviously non-decreasing.
For each stopping time $\tau$, denote $f_n^\tau= f_{n\wedge\tau}$.
It is easy to see that
$$f_n=\sum_{k\in\mathbb{Z}}(f_n^{\tau_{k+1}}-f_n^{\tau_{k}}).$$
Let
$$\mu_k=3\cdot2^k \left\|\chi_{\{ \tau_k<\infty \}} \right\|_{p(\cdot)},\quad\mbox{and}\quad a_n^k= \frac{ f_n^{\tau_{k+1}}-f_n^{\tau_{k}} }{\mu_k}.$$
If $\mu_k=0$ then let $a_n^k=0$ for all
$k\in\mathbb{Z},n\in\mathbb{N}$. Then $(a_n^k)_{n \geq 0}$ is a martingale for each fixed $k\in
\mathbb{Z}$. Since $s(f^{\tau_k}) = s_{\tau_k}(f) \leq 2^k$, we get
$$ s\left((a_n^k)_{n \geq 0} \right) \leq \frac{s(f^{\tau_{k+1}})+s(f^{\tau_k})}{\mu_k} \leq
\left\| \chi_{\{ \tau_k<\infty \}} \right\|_{p(\cdot)}^{-1}.$$ Hence
it is easy to check that $(a_n^k)_{n \geq 0}$ is a bounded $L_2$-martingale. Consequently, there exists an element $a^k \in L_2$ such
that $\E_{\mathcal{F}_n}a^k= a_n^k$. If $n \leq \tau_k$, then $a_n^k=0$, and $s(a^k) \leq  \left\|
\chi_{\{ \tau_k<\infty \}} \right\|_{p(\cdot)}^{-1}$. Thus we
conclude that $a^k$ is really a $(1,p(\cdot),\infty)$-atom.

Denote $\mathcal{O}_k = \{\tau_k < \infty\} = \{ s(f)>2^k \}$.
Recalling that $\tau_k$ is non-decreasing for each $k \in
\mathbb{Z}$, we have $\mathcal{O}_k \supset \mathcal{O}_{k+1}$. Then
$$\sum_{k \in \mathbb{Z}} \left({3\cdot2^k \chi_{\mathcal{O}_k}(x)}\right)^{\underline{p}}$$
is the sum of the geometric sequence $\left\{\left({3\cdot2^k
\chi_{\mathcal{O}_k}(x)}\right)^{\underline{p}}\right\}_{{k \in
\mathbb{Z}}}$. Thus, we can claim that
$$\sum_{k \in \mathbb{Z}} \left({3\cdot2^k \chi_{\mathcal{O}_k}(x)} \right)^{\underline{p}} \thickapprox
\left(\sum_{k \in \mathbb{Z}} {3\cdot2^k
\chi_{\mathcal{O}_k}(x)}\right)^{\underline{p}}   \thickapprox
\left(\sum_{k \in \mathbb{Z}} {3\cdot2^k \chi_{\mathcal{O}_k
\backslash \mathcal{O}_{k+1}}(x)}\right)^{\underline{p}}.
$$Indeed, for each fixed $x_0 \in \Omega$, there is $k_0 \in
\mathbb{Z}$ such that $x_0 \in \mathcal{O}_{k_0}$ but $\not\in
\mathcal{O}_{k_0+1}$, then
\begin{eqnarray*}
\sum_{k=-\infty}^{k_0} \left({3\cdot2^k
\chi_{\mathcal{O}_k}(x_0)}\right)^{\underline{p}} &=&
\sum_{k=-\infty}^{k_0}\left(3\cdot2^k
\right)^{\underline{p}}
=\left( 3\cdot2^{k_0} \right)^{\underline{p}} \frac{1}{1-2^{-{\underline{p}} }}
\\&\lesssim& \left( 3\cdot2^{k_0} \right)^{\underline{p}} \left(\frac{1}{1-\frac{1}{2}} \right)^{\underline{p}}
\\&=& \left(\sum_{k=-\infty}^{k_0} {3\cdot2^k
\chi_{\mathcal{O}_k}(x_0)}\right)^{\underline{p}}
\\&\lesssim& \left(\sum_{k=-\infty}^{k_0} {3\cdot2^k
\chi_{\mathcal{O}_k\backslash
\mathcal{O}_{k+1}}(x_0)}\right)^{\underline{p}}.
\end{eqnarray*}
Thus
\begin{eqnarray*}
\mathcal{A}(\{ \mu_k\},\{a^k\},\{ \tau_k\})&=&\Bigg\|\Bigg\{\sum_{k\in\mathbb{Z}}\Bigg(\frac{\mu_k\chi_{\{\tau_k<\infty\}}}
{\|\chi_{\{\tau_k<\infty\}}\|_{p(\cdot)}}\Bigg)^{\underline{p}}\Bigg\}^{\frac{1}{\underline{p}}}\Bigg\|_{p(\cdot)}
\\&=&\Bigg\|\Bigg\{\sum_{k\in\mathbb{Z}}\Bigg(3\cdot2^k\chi_{\{\tau_k<\infty\}}
\Bigg)^{\underline{p}}\Bigg\}^{\frac{1}{\underline{p}}}\Bigg\|_{p(\cdot)}
\\&\lesssim&\left\|\sum_{k \in
\mathbb{Z}} {3\cdot2^k \chi_{\mathcal{O}_k \backslash
\mathcal{O}_{k+1}}}\right\|_{p(\cdot)}
\\&=& \inf \Big\{ \lambda >0: \int_\Omega \Big(\sum_{k \in \mathbb{Z}} \frac{3\cdot2^k \chi_{\mathcal{O}_k \backslash \mathcal{O}_{k+1}}(x)}{\lambda} \Big)^{p(x)} d\mathbb{P} \leq 1 \Big\}
\\&=& \inf \Big\{ \lambda >0: \sum_{k \in \mathbb{Z}}
\int_{\mathcal{O}_k \backslash \mathcal{O}_{k+1}} \left(
\frac{3\cdot2^k }{\lambda} \right)^{p(x)} d\mathbb{P} \leq 1 \Big\}
\\&\approx& \inf \left\{ \lambda >0:\int_\Omega \left( \frac{s(f)}{\lambda} \right)^{p(x)} d\mathbb{P} \leq 1 \right\}.
\end{eqnarray*}
Therefore, we obtain
$$\mathcal{A}(\{ \mu_k\},\{a^k\},\{ \tau_k\}) \lesssim \left\|s(f)\right\|_{p(\cdot)} = \left\|f \right\|_{H_{p(\cdot)}^s}.$$

\noindent We now verify the sum $\sum_{k\in \mathbb{Z}} \mu_k a^k$ converges in ${H_{p(\cdot)}^s}$. By the equality
$s(f-f^{\tau_k})^2=s(f)^2-s(f^{\tau_k})^2$ we have
$$s(f-f^{\tau_k}),~s(f^{\tau_k}) \leq s(f) \quad\mbox{and}\quad s(f-f^{\tau_k}),~s(f^{\tau_{-k}})\rightarrow 0\quad\mbox{a.e., as $k\rightarrow\infty$} .$$
Consequently, by the dominated convergence theorem in variable $L^{p(\cdot)}$ (Theorem 2.62 in \cite{CF})
$$\left\|f-\sum_{k=-M}^N\mu_ka^k \right\|_{H_{p(\cdot)}^s}^{\underline{p}} \leq
\left\|f-f^{\tau_{N+1}}\right\|_{H_{p(\cdot)}^s}^{\underline{p}} + \left\|f^{\tau_{-M}} \right\|_{H_{p(\cdot)}^s}^{\underline{p}} $$
converges to $0$ a.e. as $M,N\rightarrow \infty$.

Conversely, by the definition of $(1,p(\cdot),\infty)$-atom, we have almost everywhere
$$s(a)=s(a)\chi_{\{\tau<\infty\}} \leq \left\|s(a)\right\|_\infty\chi_{\{\tau<\infty\}}
\leq
\left\|\chi_{\{\tau<\infty\}}
\right\|_{p(\cdot)}^{-1}\chi_{\{\tau<\infty\}},$$
where $a$ is a $(1,p(\cdot),\infty)$-atom. By the subadditivity of the
conditional quadratic variation operator, we obtain that
$$s(f)\leq \sum_{k\in\mathbb{Z}}\mu_k s(a^k) \leq
\sum_{k\in\mathbb{Z}}\mu_k\frac{\chi_{\{\tau_k<\infty\}}}{
\left\|\chi_{\{\tau_k<\infty\}} \right\|_{p(\cdot)}}.$$ Thus
\begin{eqnarray*}
\|f\|_{H_{p(\cdot)}^s}=\|s(f)\|_{p(\cdot)} &\leq&
\left\|\sum_{k\in\mathbb{Z}}\mu_k\frac{\chi_{\{\tau_k<\infty\}}}{\|\chi_{\{\tau_k<\infty\}}\|_{p(\cdot)}}
\right\|_{p(\cdot)}
\\&\leq& \left\|\left\{\sum_{k\in\mathbb{Z}}\left(\mu_k\frac{\chi_{\{\tau_k<\infty\}}}{\|\chi_{\{\tau_k<\infty\}}\|_{p(\cdot)}}\right)^{\underline{p}}\right\}
^{\frac{1}{\underline{p}}}\right\|_{p(\cdot)}
\\&=&\mathcal{A}(\{\mu_k\},\{a^k\},\{\tau_k\}).
\end{eqnarray*}
Hence we can conclude that
$\|f\|_{H_{p(\cdot)}^s} \thickapprox \|f\|_{H_{p(\cdot)}^{s, at}}$ and
the proof is complete now.
\end{proof}

\begin{remark}
It is showed in Theorem 5.1 in \cite{HO2} that, for the atomic decomposition of Hardy-Morrey spaces with variable exponents $p(\cdot)$ on $\mathbb{R}^n$, the exponent function $p(\cdot)$ is
not necessary to be log-H\"{o}lder continuous.
\end{remark}

\section{The duality and John-Nirenberg theorem}

In this section we establish the dual space of $H_{p(\cdot)}^s$ by the atomic decomposition established in Section 4 and prove the John-Nirenberg inequalities in the setting of variable exponents.

\begin{proposition} \label{proposition51} Let $p(\cdot)\in \mathcal{P}$ satisfy \eqref{65} with $0<p_-\leq p_+<\infty$.
\begin{enumerate}
\item If $q(\cdot)\in \mathcal{P}$ satisfies \eqref{65}, then $p(\cdot)+ q(\cdot)$ also satisfies \eqref{65};\\
\item $\frac{1}{p(\cdot)}$ satisfy \eqref{65};\\
\item If $\frac{1}{p(x)}+\frac{1}{q(x)}=1$, then $q(\cdot)$ satisfies \eqref{65};\\
\item If $q(\cdot)\in \mathcal{P}$ satisfies \eqref{65} and $\frac{1}{p(x)}+\frac{1}{q(x)}=\frac{1}{r(x)}$, then $r(\cdot)$ satisfies \eqref{65}.
\end{enumerate}
\end{proposition}
\begin{proof} (1) Set $h(\cdot)=p(\cdot)+ q(\cdot)$, then $$h_-(A)-h_+(A)\geq p_-(A)+q_-(A)-p_+(A)-q_+(A).$$
Hence $$\mathbb{P}(A)^{h_-(A)-h_+(A)}\leq \mathbb{P}(A)^{p_-(A)-p_+(A)+q_-(A)-q_+(A)} \leq K_{p(\cdot)}K_{q(\cdot)}\triangleq K.$$

(2) We have
$$\mathbb{P}(A)^{1/p_+(A)-1/p_-(A)}=\mathbb{P}(A)^{\frac{p_-(A)-p_+(A)}{p_+(A)p_-(A)}}\leq K_{p(\cdot)}^{\frac{1}{p_+(A)p_-(A)}}.$$
If $p_-(\Omega)\geq 1$, then $K_{p(\cdot)}^{\frac{1}{p_+(A)p_-(A)}} \leq K_{p(\cdot)}.$
If $0<p_-(\Omega)<1$, then $$K_{p(\cdot)}^{\frac{1}{p_+(A)p_-(A)}} \leq K_{p(\cdot)}^{1/p_-^2(\Omega)}\triangleq K.$$

(3) Set $h(\cdot)=1-\frac{1}{p(\cdot)}$. We get
$$\mathbb{P}(A)^{h_-(A)-h_+(A)}=\mathbb{P}(A)^{1-1/p_-(A)-1+1/p_+(A)} \leq K_{p(\cdot)}^{\frac{1}{p_+(A)p_-(A)}}\leq K_{p(\cdot)}^{1/p_-^2(\Omega)} \triangleq K.$$
Hence we have $1-\frac{1}{p(\cdot)}$ satisfies \eqref{65}. Using (2), we get desired result.

(4) It follows from (1) and (2). The proof is complete.
\end{proof}

It is easy to prove that for all $B\in \mathcal{F}$
 $$ \mathbb{P}(B)^{p_-(B)-p(x)}\quad\mbox{\big(and  $\mathbb{P}(B)^{p(x)-p_+(B)}\big) $}\leq K \quad \forall x\in B,$$
 if $p(\cdot)$ satisfies \eqref{65}. Using this result, we have the following lemma.

\begin{lemma} \label{lemma52}  Let $p(\cdot)\in \mathcal{P}$ and satisfy \eqref{65} and $0< p_-\leq p_+ < \infty$. Then
for all set $B\in \mathcal{F}$, we have
$$\mathbb{P}(B)^{1/{p_-(B)}}  \approx  \mathbb{P}(B)^{1/{p(x)}} \approx \mathbb{P}(B)^{1/{p_+(B)}} \approx  \|\chi_B\|_{p(\cdot)}\quad \forall x\in B.$$ \end{lemma}

\begin{proof} Obviously, we have $\mathbb{P}(B)^{1/{p_-(B)}}  \leq \mathbb{P}(B)^{1/{p(x)}} \leq \mathbb{P}(B)^{1/{p_+(B)}}$, for all $x\in B$. Since \eqref{65}, we have
\begin{eqnarray*}
\frac{\mathbb{P}(B)^{1/{p(x)}}}{\mathbb{P}(B)^{1/{p_-(B)}}} \leq \mathbb{P}(B)^{\frac{p_-(B)-p(x)}{p_-(B)p(x)}}\leq K_{p(\cdot)}^{\frac{1}{p_-^2(\Omega)}}\triangleq K.
\end{eqnarray*}
This implies $\mathbb{P}(B)^{1/{p(x)}} \leq K \mathbb{P}(B)^{1/{p_-(B)}}$.

Then it is easy to check that $\mathbb{P}(B)^{1/{p_-(B)}}  \approx  \mathbb{P}(B)^{1/{p(x)}} \approx \mathbb{P}(B)^{1/{p_+(B)}}$. And we
have
$$\frac{\chi_B(x)}{\mathbb{P}(B)^{1/{p_-(B)}}} \approx  \frac{\chi_B(x)}{\mathbb{P}(B)^{1/{p(x)}}},$$
that is
$$\left(\frac{\chi_B(x)}{\mathbb{P}(B)^{1/{p_-(B)}}}\right)^{p(x)} \geq \frac{\chi_B(x)}{\mathbb{P}(B)}\geq \left(\frac{\chi_B(x)}{K \mathbb{P}(B)^{1/{p_-(B)}}}\right)^{p(x)}.$$
So
$$\int_\Omega \left(\frac{\chi_B(x)}{\mathbb{P}(B)^{1/{p_-(B)}}}\right)^{p(x)} d\mathbb{P} \approx  \int_\Omega \frac{\chi_B(x)}{\mathbb{P}(B)} d\mathbb{P} =1. $$
Consequently, $\|\chi_B\|_{p(\cdot)} \approx \mathbb{P}(B)^{1/{p_-(B)}}$ and we get the desired result.
\end{proof}

\begin{remark} Lemma \ref{lemma52} is also true for $p_+=\infty.$ In this case, we need to employ a slightly different definition of $\|\cdot\|_{p(\cdot)}$; see Definition 2.16 in \cite{CF}.

\end{remark}

\begin{corollary} \label{corollary53}Let $p(\cdot) \in \mathcal{P}$ satisfy \eqref{65} with $0<p_-\leq p_+<\infty$.
\begin{enumerate}
\item Then for all set $B\in \mathcal{F}$, we have $$\|\chi_B\|_1 \approx \|\chi_B\|_{p(\cdot)}\|\chi_B\|_{q(\cdot)},$$
where
$$1= \frac{1}{p(x)}+\frac{1}{q(x)}.$$

\item Let $q(\cdot) \in \mathcal{P}$ and satisfies \eqref{65}. Then
for all set $B\in \mathcal{F}$, we have
$$\|\chi_B\|_{r(\cdot)} \approx \|\chi_B\|_{p(\cdot)}\|\chi_B\|_{q(\cdot)},$$
where
$$\frac{1}{r(x)}= \frac{1}{p(x)}+\frac{1}{q(x)}.$$
\end{enumerate}
\end{corollary}

\begin{proof} It follows from Proposition \ref{proposition51} and Lemma \ref{lemma52} that
$$\|\chi_B\|_{r(\cdot)}\approx \mathbb{P}(B)^{\frac{1}{r(x)}}=\mathbb{P}(B)^{{\frac{1}{p(x)}+\frac{1}{q(x)}}}\approx\|\chi_B\|_{p(\cdot)}\|\chi_B\|_{q(\cdot)}
, \quad \forall x\in B.$$
\end{proof}

As application of atomic decomposition, we now prove
a duality theorem. First let
us introduce the new Lipschitz spaces with variable exponents.

\begin{definition}   Given $1/\alpha(\cdot)$ is a
variable exponent ($1/\alpha(\cdot)=\infty$ is allowed) and a
constant $1\leq q <\infty$. Define $\Lambda_q(\alpha(\cdot))$ as the
space of functions $f \in L^q$ for which
$$\|f\|_{\Lambda_q(\alpha(\cdot))} = \sup_{\tau\in \mathcal{T}} \left\| \chi_{\{ \tau<\infty \}} \right\|_{\frac{1}{\alpha(\cdot)}}^{-1}  \left\| \chi_{\{ \tau<\infty \}} \right\|_{q}^{-1}  \| f-f^\tau\|_q$$
is finite.
\end{definition}

\begin{theorem} Given $p(\cdot) \in \mathcal{P},~0<p_- \leq p_+
\leq1$ and $p(\cdot)$ satisfies \eqref{65}. Then
$$\Big(H_{p(\cdot)}^s\Big)^*=\Lambda_2(\alpha(\cdot)),\quad \alpha(x) =1/{p\left(x\right)}-1.$$
\end{theorem}

\begin{proof}  We first claim that $\alpha(\cdot)$ satisfies \eqref{65} by Proposition \ref{proposition51}(1). Let $\varphi \in
\Lambda_2(\alpha(\cdot)) \subset L^2$ and for all $f\in L^2$, define
$$l_\varphi(f)=\mathbb{E}(f \varphi).$$ We shall show that $l_\varphi$ is a
bounded linear functional on $H_{p(\cdot)}^s$. By Theorem \ref{theorem42}, we know that $L^2$ is
dense  in $H_{p(\cdot)}^s$. Take the same stopping
times $\tau_k$, atoms $a^k$ and nonnegative numbers $\mu_k$ as we
do in Theorem \ref{theorem42}. It follows from Theorem \ref{theorem42} that $f=\sum_{k\in
\mathbb{Z}} \mu_k a^k \;(\forall f \in L_2)$. Hence
$$l_\varphi(f)=\mathbb{E}(f \varphi) = \sum_{k\in \mathbb{Z}} \mu_k \mathbb{E}(a^k \varphi).$$
By the definition of the atom $a^k$, $\mathbb{E}(a^k \varphi) =
\mathbb{E}(a^k (\varphi- \varphi^{\tau_k})) $ always holds. It
follows from Corollary \ref{corollary53} that
$$\left\|\chi_{\{\tau_k<\infty\}} \right\|_{p(\cdot)}  \thickapprox \left\|\chi_{\{\tau_k<\infty\}}\right\|_{\frac{1}{\alpha(\cdot) }} \left\|\chi_{\{\tau_k<\infty\}}\right\|_2 \left\|\chi_{\{\tau_k<\infty\}}\right\|_2.$$
Thus, using H\"{o}lder's inequality we can conclude that
\begin{eqnarray*}
|l_\varphi(f)| &\leq& \sum_{k\in \mathbb{Z}} \mu_k \int_\Omega |a^k|
|\varphi- \varphi^{\tau_k}| d\mathbb{P}
\\&\leq& \sum_{k\in \mathbb{Z}} \mu_k \|a^k\|_2 \|\varphi- \varphi^{\tau_k}\|_2
\\&\leq& \sum_{k\in \mathbb{Z}} \mu_k \frac{|\{\tau_k<\infty\}|^{\frac{1}{2}}}{\left\|\chi_{\{\tau_k<\infty\}} \right\|_{p(\cdot)} } \|\varphi- \varphi^{\tau_k}\|_2
\\&\lesssim& \sum_{k\in \mathbb{Z}} \mu_k \|\varphi\|_{\Lambda_2(\alpha(\cdot))}.
\end{eqnarray*}
Then, we obtain from Proposition \ref{proposition41} and Theorem \ref{theorem42} that
$$|l_\varphi(f)| \lesssim \|f\|_{H_{p(\cdot)}^s} \|\varphi\|_{\Lambda_2(\alpha(\cdot))}.$$
Consequently, $l_\varphi$ can be
extended to $H_{p(\cdot)}^s$ uniquely.

On the other hand, let $l$ be an arbitrary bounded linear functional
on $H_{p(\cdot)}^s$. We shall show that there exists $\varphi \in
\Lambda_2(\alpha(\cdot))$ such that $l=l_\varphi$ and
$$\|\varphi\|_{\Lambda_2(\alpha(\cdot))} \lesssim \|l\| .$$
Since $0<p_- \leq p_+ \leq1$,  thus it follows from Lemma \ref{lemma21} and Theorem 2.8 in \cite{KR} that
\begin{eqnarray*}
\|f\|_{H_{p(\cdot)}^s} &=& \|s(f)\|_{p(\cdot)}=\|s(f)^{p_-}\|_{\frac{p(\cdot)}{p_-}}^{\frac{1}{p_-}}
\\&\leq & \Big(2\|s(f)^{p_-}\|_{\frac{2}{p_-}}\Big)^{\frac{1}{p_-}} =2^{\frac{1}{p_-}}\|s(f)\|_{2}=2^{\frac{1}{p_-}}\|f\|_{2},\quad \forall f\in L^2.
\end{eqnarray*}
Then the space $L^2$ can be embedded continuously in $H_{p(\cdot)}^s$.
Consequently, there exists $\varphi \in L^2$ such
that
$$l(f)=\mathbb{E}(f \varphi),~~~~~~\forall f\in L^2.$$
Let $\tau$ be an arbitrary stopping time and
$$g = \frac{\varphi -\varphi^\tau } { \|\varphi- \varphi^\tau\|_2  \left\| \chi_{\{ \tau<\infty \}} \right\|_{\frac{1}{\alpha(\cdot)}}  \left\|\chi_{\{ \tau<\infty \}}\right\|_2 } .$$
Then $g$ is not necessarily a $(1,p(\cdot),\infty)$-atom but it
satisfies $(1)$ in Definition \ref{definition11}, thus we have
$$s(g) = s(g) \chi_{\{ \tau<\infty \}}.$$
Since $$\frac{1}{p(x)} = \frac{1}{2}+ \frac{1}{1/\alpha(x)}+
\frac{1}{2},$$ then by H\"{o}lder's inequality we get
\begin{eqnarray*}
\|g\|_{H_{p(\cdot)}^s} &=& \frac{\|s(\varphi
-\varphi^\tau)\|_{p(\cdot)} } { \left\|\varphi- \varphi^\tau
\right\|_2 \left\| \chi_{\{ \tau<\infty \}}
\right\|_{\frac{1}{\alpha(\cdot)}} \left\|\chi_{\{ \tau<\infty
\}}\right\|_2 }
\\&\lesssim& \frac{\|s(\varphi -\varphi^\tau)\|_2 \left\| \chi_{\{ \tau<\infty \}} \right\|_{\frac{1}{\alpha(\cdot)}} \left\| \chi_{\{ \tau<\infty \}} \right\|_2 }
{ \|\varphi- \varphi^\tau\|_2 \left\|\chi_{\{ \tau<\infty \}}
\right\|_{\frac{1}{\alpha(\cdot)}} \left\|\chi_{\{ \tau<\infty
\}}\right\|_2 }
\\&=& 1.
\end{eqnarray*}
Thus
\begin{eqnarray*}
\|l\| \gtrsim l(g) &=& \mathbb{E}\left(g(\varphi-
\varphi^\tau)\right)
\\&=& \left\| \chi_{\{ \tau<\infty \}} \right\|_{\frac{1}{\alpha(\cdot)}}^{-1}  \left\| \chi_{\{ \tau<\infty \}} \right\|_{2}^{-1}  \| \varphi-\varphi^\tau\|_2
\end{eqnarray*}
and we get that $\|\varphi\|_{\Lambda_2(\alpha(\cdot))} \lesssim
\|l\|$ and the proof is complete.
\end{proof}

 We now turn to the John-Nirenberg theorem with variable exponents. Recall that $BMO_p(1\leq p <\infty)$ is the space of those functions $f$ for which
$$\|f\|_{BMO_{p}}=\sup_{\tau \in \mathcal{T} } \|\chi_{\{\tau<\infty\}}\|_p^{-1} \|f-f^{\tau-1}\|_p<\infty.$$

\begin{definition}  Given $p(\cdot) \in \mathcal{P}$ and $\mathcal{T}$ be the sets of all stopping times relative to $\{\mathcal{F}_n\}_{n\geq0}$. Define
$$BMO_{p(\cdot)}=\left\{ f=(f_n)_{n\geq0}: \|f\|_{BMO_{p(\cdot)}}< \infty \right\},$$
where
$$\|f\|_{BMO_{p(\cdot)}}=\sup_{\tau \in \mathcal{T} } \left\|\chi_{\{\tau<\infty\}}\right\|_{p(\cdot)}^{-1} \|f-f^{\tau-1}\|_{p(\cdot)}.$$
\end{definition}

\begin{lemma} \label{lemma57}(see \cite{FW1}) If $1\leq p <\infty$, then
$$\|f\|_{BMO_1} \thickapprox \|f\|_{BMO_{p}}.$$ \end{lemma}

\begin{proposition}\label{JN}  If $p(\cdot)\in\mathcal{P}$ satisfies \eqref{65} and $1\leq p_- \leq p_+ < \infty$, then we have that for all
$f\in BMO_1$
$$\|f \|_{BMO_1}\lesssim \|f \|_{BMO_{p(\cdot)}} \lesssim \|f \|_{BMO_1}.$$
\end{proposition}
\begin{proof} By H\"{o}lder's inequality and Corollary \ref{corollary53}, we have that
\begin{eqnarray*}
\frac{\|f -f^{\tau-1}\|_1}{\left\|\chi _{\{\tau < \infty\}} \right\|_1}
&\lesssim&
\frac{\left\|f -f^{\tau-1}\right\|_{p(\cdot)}\left\|\chi _{\{\tau < \infty\}} \right\|_{p'(\cdot)}}{\left\|\chi _{\{\tau < \infty\}} \right\|_1}
\\&=&
\frac{\left\|f -f^{\tau-1}\right\|_{p(\cdot)}}{\left\|\chi _{\{\tau < \infty\}} \right\|_{p(\cdot)}} \cdot \frac{\left\|\chi _{\{\tau < \infty\}} \right\|_{p(\cdot)}\left\|\chi _{\{\tau < \infty\}} \right\|_{p'(\cdot)}}{\left\|\chi _{\{\tau < \infty\}} \right\|_1}
\\&\leq&
C_{p(\cdot)} \|f\|_{BMO_{p(\cdot)}},
\end{eqnarray*}
where $$\frac{1}{p(x)}+\frac{1}{p'(x)}=1.$$
Hence $\|f\|_{BMO_1} \lesssim \|f\|_{BMO_{p(\cdot)}}$.

Since
\begin{eqnarray*}
\|f -f^{\tau-1}\|_{p(\cdot)} & \lesssim & \|f -f^{\tau-1}\|_{p_+}
\left\|\chi_{\{\tau < \infty\}} \right\|_{\frac{p_+ p(\cdot)}{p_+ -p(\cdot)}}
\\&=&
\frac{ \|f -f^{\tau-1}\|_{p_+} } { \|\chi _{\{\tau < \infty\}}\|_{p_+} }  \left\|\chi _{\{\tau < \infty\}} \right\|_{\frac{p_+p(\cdot)} {p_+ -p(\cdot)}} \|\chi _{\{\tau < \infty\}}\|_{p_+},
\end{eqnarray*}then by Lemma \ref{lemma57}, we get
$$\|f -f^{\tau-1}\|_{p(\cdot)}\lesssim \|f\|_{BMO_1}\left\|\chi _{\{\tau < \infty\}}\right\|_{\frac{p_+p(\cdot)} {p_+ -p(\cdot)}} \left\|\chi _{\{\tau < \infty\}}\right\|_{p_+}.$$
Thus by Corollary \ref{corollary53}
\begin{eqnarray*}
\frac{\left\|f -f^{\tau-1}\right\|_{p(\cdot)}}{\left\|\chi _{\tau < \infty} \right\|_{p(\cdot)}} & \lesssim &
\left\|f\right\|_{BMO_1} \left\|\chi _{\{\tau < \infty\}}\right\|_{\frac{p_+p(\cdot)} {p_+ -p(\cdot)}} \left\|\chi _{\{\tau < \infty\}}\right\|_{p_+} \left\|\chi _{\{\tau < \infty\}} \right\|_{p(\cdot)}^{-1}
\\&\lesssim &
\|f\|_{BMO_1}.
\end{eqnarray*}
This means $$\|f \|_{BMO_{p(\cdot)}} \lesssim \|f \|_{BMO_1}.$$ \end{proof}

By applying Proposition \ref{JN}, we prove the following exponential integrability form of the John-Nirenberg theorem, which should be compared with the very recent result, Theorem 3.2 in \cite{HO}.

\begin{theorem}  Let $p(\cdot) \in \mathcal{P}$ satisfy \eqref{65} and $1\leq p_-\leq p_+<\infty$, then there exist constants $C_1, C_2 >0$ such that for every $f\in BMO_1$ and $\tau \in \mathcal{T}$,
$$\Big\| \chi_{\{\tau<\infty\}\cap\{f-f_{\tau-1}\geq t\} }\Big\|_{p(\cdot)} \leq C_1 e^{-\frac{C_2 t}{\|f\|_{BMO_1}}}\|\chi_{\{\tau<\infty\}}\|_{p(\cdot)}\quad t>0.$$
\end{theorem}
\begin{proof}  Using Lemma \ref{lemma21} and Theorem \ref{JN}, we point out that for $r\geq 1,$
$$\sup_\tau \frac{\||f-f^{\tau-1}|^r\|_{p(\cdot)}^{1/r}}{\|\chi_{\{\tau<\infty\}}\|_{p(\cdot)}^{1/r}} =\|f\|_{BMO_{rp(\cdot)}}\leq C \|f\|_{BMO_1} \triangleq C_0.$$
This implies that
$$\||f-f^{\tau-1}|^r\|_{p(\cdot)}\leq C_0^r \|\chi_{\{\tau<\infty\}}\|_{p(\cdot)}.$$
Then we get that
$$\Big\| \chi_{\{\tau<\infty\}\cap\{f-f_{\tau-1}\geq t\} }\Big\|_{p(\cdot)} \leq \frac{1}{t^r}\||f-f^{\tau-1}|^r \chi_{\{\tau<\infty\}}\|_{p(\cdot)} \leq \frac{C_0^r}{t^r} \|\chi_{\{\tau<\infty\}}\|_{p(\cdot)}.$$
If $t\geq 2C_0$, we take $r=\frac{t}{2C_0} \geq 1$, then
$$\big(\frac{C_0}{t}\big)^r \leq \frac{1}{2^r}=e^{-r\ln2}=e^{-\frac{t}{2C_0}\ln2}=e^{-\frac{t}{2C\|f\|_{BMO_1} }\ln2}=e^{-\frac{C_2 t}{\|f\|_{BMO_1}}},$$
where $C_2=\frac{1}{2C}\ln2$.

If $ t<2C_0$, take $C_2=\frac{1}{2C}\ln2$. Then $e^{-\frac{C_2 t}{\|f\|_{BMO_1}}} = \big(\frac{1}{2}\big)^{\frac{t}{2C_0}}>1/4.$
Since $$\{\tau<\infty\}\cap\{f-f_{\tau-1}\geq t\} \subset \{\tau<\infty\},$$ it follows that
$$\Big\| \chi_{\{\tau<\infty\}\cap\{f-f_{\tau-1}\geq t\} }\Big\|_{p(\cdot)} \leq \|\chi_{\{\tau<\infty\}}\|_{p(\cdot)}\leq 4 e^{-\frac{C_2 t}{\|f\|_{BMO_1}}}\|\chi_{\{\tau<\infty\}}\|_{p(\cdot)}.$$
We conclude this proof.
\end{proof}

\begin{remark}
The result above depends on condition \eqref{65}, and we refer to Corollary 3.5 in \cite{HO3} for another John-Nirenberg theorem with a non-log-H\"{o}lder exponent function $p(\cdot)$ on $\mathbb{R}^n$.
\end{remark}
\begin{remark}
Recently, there are some new results concerning martingale Hardy spaces with variable exponents; see \cite{haojiao,LPD,wujiao}.

\end{remark}

{\bf Acknowledgements.} The authors are grateful to the referee for her/his carefully reading and useful comments and suggestions. Yong Jiao is supported by NSFC(11471337) and Hunan Provincial Natural Science Foundation(14JJ1004); Wei Chen is supported by NSFC(11101353).

\bibliographystyle{amsplain}

\end{document}